\documentclass{article}

\usepackage{amsmath}
\usepackage{amssymb}
\usepackage{amsthm}
\usepackage{enumerate}
\usepackage{tikz-cd}
\usepackage{pgf,tikz}
\usetikzlibrary{arrows}
\usepackage{graphics}
\usepackage{commath}
\usepackage{epstopdf}
\usepackage{verbatim}
\usepackage{mathrsfs}
\usepackage{breqn}
\usepackage{authblk}
\usepackage{fancyhdr}

\usepackage{geometry}
\usepackage[pagebackref=true,hypertexnames=false]{hyperref}

\title{$C^0$ Lagrangian monodromy\footnote{Mathematics Subject Classification (MSC) codes (2020): 53E50, 57R17, 37J39.}}
\author{Noah Porcelli}

\newcommand{\Addresses}{{
  \bigskip
  \footnotesize

  \textsc{Imperial College Department of Mathematics, Huxley Building, 180 Queen's Gate, London SW7 2RH, U.K.}\par\nopagebreak
  \textit{n.porcelli@imperial.ac.uk}

}}

\date{}
   
\usepackage{hyperref}
\usepackage{cleveref}
\newtheorem{TT}{Theorem}[section]
\newtheorem{LL}[TT]{Lemma}
\newtheorem{CC}[TT]{Corollary}
\newtheorem{PP}[TT]{Proposition}
\newtheorem{DD}[TT]{Definition}
\newtheorem{EE}[TT]{Example}
\newtheorem{RR}[TT]{Remark}
\newtheorem{QQ}[TT]{Question}
\newtheorem{BB}[TT]{Assumption}

\newtheorem{JJ}[TT]{Conjecture}

\newcommand{\Dim}{\mathrm{Dim}}

\newcommand{\Id}{\mathrm{Id}}
\newcommand{\R}{\mathbb{R}}
\newcommand{\C}{\mathbb{C}}

\def\bC{\mathbb{C}}

\def\bR{\mathbb{R}}

\def\bZ{\mathbb{Z}}

\newcommand{\cM}{\mathcal{M}}

\newcommand{\cP}{\mathcal{P}}

\newcommand{\cU}{\mathcal{U}}

\newcommand{\eps}{\varepsilon}

\begin{document}
    \maketitle
    \pagenumbering{arabic}
    \begin{abstract}
        We prove that (under appropriate orientation assumptions), the action of a Hamiltonian homeomorphism $\phi$ on the cohomology of a relatively exact Lagrangian fixed by $\phi$ is the identity. This extends results of Hu-Lalonde-Leclercq \cite{Hu-Lalonde-Leclercq} and the author \cite{Porcelli} in the setting of Hamiltonian diffeomorphisms. We also prove a similar result regarding the action of $\phi$ on relative cohomology.
    \end{abstract}
\section{Introduction}

\subsection{Background}

Recall a \emph{symplectic manifold} $(M, \omega)$ is a manifold $M$, equipped with a 2-form $\omega \in \Omega^2(M)$ which is closed ($d\omega = 0$) and non-degenerate (meaning $\omega^{\Dim(M)/2}$ is everywhere non-zero, in particular $\Dim(M)$ must be even). Important examples include cotangent bundles $T^*Q$ and smooth complex quasiprojective varieties. Two of the main objects of study in classical symplectic geometry are Lagrangian submanifolds and Hamiltonian diffeomorphisms along with their topological and dynamical aspects.

A \emph{Lagrangian submanifold} of a symplectic manifold $L \subseteq (M, \omega)$ is a submanifold satisfying $\omega|_L = 0$, and $\Dim(L) = \Dim(M)/2$. Examples include the zero section $Q_0$ in a cotangent bundle $T^*Q$, the real loci of complex quasiprojective varieties defined by real polynomials, and the fibres of moment maps of smooth toric varieties.

A \emph{Hamiltonian diffeomorphism} $\phi^1: M \to M$ is a diffeomorphism of $M$ which is the time-one flow of a Hamiltonian vector field $\{X^t_H\}_{t \in [0,1]}$, which is defined via the equation $\omega(\cdot, X^t_H) = dH^t$ for some smooth map $H: [0,1] \times M \to \bR$ (nondegeneracy of $\omega$ implies this is well-defined). These arise naturally in classical dynamics as a formulation of Hamiltonian's equations of motion to abstract manifolds. We write $\operatorname{Ham}(M)$ for the space of Hamiltonian diffeomorphisms of $M$. We refer to \cite{MD:Intro} for more background on symplectic geometry.

In \cite{Mei-LinYau}, Yau studied the following question on the topology of Lagrangian submanifolds:
\begin{QQ}\label{QQ1}
    Let $L \subseteq (M, \omega)$ be a Lagrangian in a symplectic manifold.
    What diffeomorphisms $\tilde \phi$ of $L$ can be extended to a Hamiltonian diffeomorphism $\phi$ on $M$?
\end{QQ}
This question has been studied in various places. For example, some partial classification results in the case $L$ is a monotone Lagrangian 2-torus have been obtained in \cite{Mei-LinYau, Jack+}; in particular, even in simple cases $\tilde \phi$ can act non-trivially on $H^*(L)$. A more comprehensive history of this question can be found in the introduction in \cite{Jack+}. 

A Lagrangian $L$ is called \emph{relatively exact} if the symplectic area of any disc with bounday on $L$ is 0: in symbols, $\omega \cdot \pi_2(M, L) = 0$. Such Lagrangians enjoy much stronger rigidity properties than arbitrary Lagrangians; for example, in many cases they cannot be displaced from themselves by a Hamiltonian isotopy \cite{Floer}.

\begin{BB}
    We now fix (for the rest of the paper) a symplectic manifold $(M, \omega)$ and a Lagrangian submanifold $L \subseteq M$. 
    
    We assume everywhere that $M$ is compact or Liouville (cf. \cite[Section 7(b)]{Seidel:book})\footnote{Roughly, this says that $M$ behaves nicely outside a compact set.}. We also assume that $L \subseteq M$ is a compact, relatively exact Lagrangian. 
    
    We also assume for simplicity that all Hamiltonians are compactly supported\footnote{This is without any loss of generality, as can be seen by using an appropriate cut-off function supported away from the image of $L$ under the flow of this Hamiltonian.}.
\end{BB}

In this setting, strong constraints have been obtained for $\tilde \phi$ to extend:

\begin{TT}\label{TT: old}
    Let $L \subseteq M$ be a compact, relatively exact Lagrangian in a (compact or Liouville) symplectic manifold $M$.

    Let $\tilde \phi: L \to L$ be a diffeomorphism, and assume that there is a Hamiltonian diffeomorphism $\phi: M \to M$ extending $\tilde \phi$.
    \begin{enumerate}
        \item \cite{Hu-Lalonde-Leclercq}: $\phi$ acts as the identity on $H^*(L; \bZ/2)$.
        \item \cite{Hu-Lalonde-Leclercq,Porcelli}: If $L$ is Spin, $\phi$ acts as the identity on $H^*(L; \bZ)$.
        \item \cite{Jack}: If $L$ is oriented, $\phi$ is orientation-preserving.
        \item \cite{Porcelli}: $\phi$ acts as the identity on the set of conjugacy classes of $\pi_1 L$.
    \end{enumerate}
\end{TT}

Further results regarding the action of $\phi$ on string topology as well as some generalised cohomology theories were obtained in \cite{Porcelli}. Motivated by these results, one can conjecture:
\begin{JJ}\label{JJ: main}
    Let $L \subseteq M$ be a compact, relatively exact Lagrangian in a (compact or Liouville) symplectic manifold $M$.

    Let $\tilde \phi: L \to L$ be a diffeomorphism which extends to a Hamiltonian diffeomorphism $\phi:M \to M$. Then $\tilde \phi$ is isotopic to the identity.
\end{JJ}
\begin{RR}
    Theorem \ref{TT: old} only constrains the homotopy class of $\tilde \phi$. Using different methods (namely generating functions), upcoming work with Courte \cite{CP} finds constraints on the smooth isotopy class of $\tilde \phi$ in the setting of Theorem \ref{TT: old} not detectable by their homotopy class, in the special case $L$ is the zero-section in the cotangent bundle of a high-dimensional torus $M = T^*T^n$.
\end{RR}

\subsection{$C^0$ symplectic topology}
\begin{DD}
    A \emph{symplectic diffeomorphism} $\theta: M \to M$ is a diffeomorphism that preserves the symplectic structure, meaning $\theta^*\omega = \omega$.
\end{DD}
Any Hamiltonian diffeomorphism is a symplectomorphism \cite[Proposition 10.2]{MD:Intro}.

$C^0$ symplectic topology studies non-smooth generalisations of Hamiltonian and symplectic diffeomorphisms:

\begin{DD}

    A \emph{symplectic homeomorphism} is a homeomorphism $\theta: M \to M$ which is a $C^0$ limit of symplectic diffeomorphisms. 

    A \emph{Hamiltonian homeomorphism} is a homeomorphism $\theta: M \to M$ which is a $C^0$ limit of Hamiltonian diffeomorphisms. We write $\overline{Ham}(M)$ for the space of Hamiltonian homeomorphisms on $M$.
\end{DD}

The field of $C^0$ symplectic topology began with the following theorem.
\begin{TT}[{Eliashberg-Gromov \cite{Eliashberg, Gromov}}]
    Any symplectic homeomorphism which is also a diffeomorphism is in fact a symplectic diffeomorphism.
\end{TT}
This is surprising since the condition of being a symplectic diffeomorphism involves the derivative of $\theta$, and in general one often loses control of derivatives under $C^0$ limits. The Hamiltonian version of this theorem, namely that any Hamiltonian homoeomorphism which is a diffeomorphism is a Hamiltonian diffeomorphism, remains open.

In general, the geometry of Hamiltonian homeomorphisms is much less understood than those of Hamiltonian diffeomorphisms; as we illustrate in the following examples, $C^0$ symplectic topology both exhibits in some ways similar rigidity phenomena to smooth symplectic topology, and in other ways much more flexibility. It is this interplay which makes it such an intriguing subject.

\begin{EE}[Rigidity in $C^0$ symplectic topology]
    Seidel and Keating \cite{Seidel:knotted,Keating:Dehn} produced large classes of examples of symplectomorphisms which are smoothly isotopic to the identity, but not isotopic to the identity through (smooth) symplectomorphisms. Roughly, these are constructed as compositions of symplectic Dehn twists. Jannaud \cite{Jannaud:free} later showed that many of these are also not isotopic to the identity through symplectic homeomorphisms.
\end{EE}

\begin{EE}[Flexibility in $C^0$ symplectic topology]
    One of the main goals in (smooth) symplectic topology has been to find lower bounds on the numbers of fixed-points of Hamiltonian diffeomorphisms $f: M \to M$. Motivated by the \emph{Arnol'd conjecture}, it has been shown \cite{Fukaya-Ono,Liu-Tian,Ruan} that if $M$ is compact and $f$ is generic, there must be at least as many fixed-points as the sum of the Betti numbers of $M$; in particular, there must be at least 2.
    
    It is natural to ask whether the same inequality holds for Hamiltonian homeomorphisms. Surprisingly, this fails badly: Buhovsky-Humilière-Seyfaddini \cite{BHS} showed that whenever $\Dim(M) \geq 4$, there exists a Hamiltonian homeomorphism $\theta: M \to M$ with a single fixed point. 
\end{EE}

\subsection{Monodromy in the $C^0$ setting}

Here we extend parts (1) and (2) of Theorem \ref{TT: old} to the $C^0$ setting (it would be interesting to extend parts (3) and (4) to the $C^0$ setting too).
\begin{TT}\label{TT: main}
    Let $L \subseteq M$ be a compact, relatively exact Lagrangian in a (compact or Liouville) symplectic manifold $M$.

    Let $\tilde \phi: L \to L$ be a homeomorphism, and assume that there is a Hamiltonian homeomorphism $\phi: M \to M$ extending $\tilde \phi$.\par
    \begin{enumerate}
        \item $\phi$ acts as the identity on $H^*(L;\bZ/2)$.
        \item If $L$ is Spin, $\phi$ acts as the identity on both cohomology $H^*(L; \bZ)$ and complex $K$-theory $K^*(L)$.
    \end{enumerate}
\end{TT}
Theorem \ref{TT: main} holds for various other generalised cohomology theories $R^*$ (such as real $K$-theory $KO^*$) provided the condition \cite[Proposition 1.14]{Porcelli} holds for $R^*$. Examples of diffeomorphisms of closed manifolds which act as the identity but not on (real or complex) $K$-theory are given in \cite[Appendices A \& B]{Porcelli}.\par
The proof of Theorem \ref{TT: main} approximates $\phi$ with a $C^0$-close Hamiltonian diffeomorphism, and proceeds by adapting the strategy of \cite{Porcelli} to this approximation.
\begin{RR}
    
    The \emph{Nearby Lagrangian conjecture} says that any closed exact Lagrangian $L$ in a cotangent bundle $T^*Q$ is Hamiltonian isotopic to the zero-section $Q_0 \subseteq T^*Q$. As we show in Section \ref{sec: NLC}, a positive resolution to this conjecture would allow us to deduce Theorem \ref{TT: main} from Theorem \ref{TT: old}.
\end{RR}
\begin{RR}
    Question \ref{QQ1} can be viewed as the study of the relative mapping class group $\pi_0 \mathrm{Ham}(M, L)$ of isotopy classes of Hamiltonian diffeomorphisms which fix $L$, and Theorem \ref{TT: main} is then the corresponding extension to Hamiltonian homeomorphisms $\pi_0 \overline{\mathrm{Ham}}(M, L)$.

    In a slightly different direction, homotopy clases of loops of Hamiltonian diffeomorphisms (equivalently, nontrivial elements of $\pi_1 \mathrm{Ham}(M)$) were constructed in \cite{Seidel:rep} and shown to be nontrivial using the \emph{Seidel representation}. This has been extended recently in \cite{HJL} to show that these loops are still nontrivial as homotopy classes of Hamiltonian homeomorphisms (equivalently, nonzero in $\pi_1 \overline{Ham}(M)$), by combining the Seidel representation with spectral invariants.
\end{RR}
\subsection{Extension to the relative setting}
    Assume $\tilde \phi: L \to L$ is a homeomorphism which extends to a Hamiltonian homeomorphism $\phi: M \to M$. $\phi$ induces a map of pairs $(M, L) \to (M, L)$, and hence a map on relative cohomology. 
    
    Conjecture \ref{JJ: main} would imply that $\phi$ is homotopic to the identity as a map of pairs; we do not prove this, but prove the cohomological analogue:
    \begin{TT}\label{TT: rel}
        Let $L \subseteq M$ be a compact, relatively exact Lagrangian in a (compact or Liouville) symplectic manifold $M$.
        
        Let $\tilde \phi: L \to L$ be a homeomorphism, and assume that there is a Hamiltonian homeomorphism $\phi: M \to M$ extending $\tilde \phi$. Then:
        \begin{enumerate}
            \item $\phi$ acts as the identity on $H^*(M, L;\bZ/2)$.
            \item If $L$ is Spin, $\phi$ acts as the identity on both $H^*(M, L; \bZ)$ and $K^*(M, L)$.
        \end{enumerate}
    \end{TT}
    \begin{RR}
        It is tempting to attempt to deduce Theorem \ref{TT: rel} the long exact sequence on cohomology associated to the pair $(M, L)$, and use Theorem \ref{TT: main} along with the fact $\phi$ must act as the identity on $H^*(M)$.
        
        However, this does not follow automatically, because the five lemma does \emph{not} constrain automorphisms.
        
        More precisely, there exists a chain complex $B_*$, sitting in an exact triangle $A_* \xrightarrow{\alpha} B_* \xrightarrow{\beta} C_*$, along with an endomorphism $\beta: B_* \to B_*$ which commutes with the identity on both $A_*$ and $C_*$, but such that $\beta$ does not act as the identity on homology.
        
        We illustrate a concrete example of this phenomenon, in the following commutative diagram:
        \begin{equation}
            \begin{tikzcd}
                0
                \ar[r]
                &
                \bZ 
                \ar[r, "{\left(\begin{smallmatrix} 1 \\ 0 \end{smallmatrix}\right)}"]
                \ar[loop below, "\Id_\bZ"]
                &
                \bZ^{\oplus 2}
                \ar[r, "{\left(\begin{smallmatrix} 0 \, 1\end{smallmatrix}\right)}"]
                \ar[loop below, "{\left(\begin{smallmatrix} 1 \, 1 \\ 0 \, 1 \end{smallmatrix}\right)}"]
                &
                \bZ
                \ar[loop below, "\Id_\bZ"]
                \ar[r]
                &
                0
            \end{tikzcd}
        \end{equation}
    \end{RR}
    The proof of Theorem \ref{TT: rel} is similar to the proof of Theorem \ref{TT: main}, but whereas the proof of Theorem \ref{TT: main} uses a moduli space of holomorphic curves with domain $Z_C$, $\cP_C$, along with the evaluation map on the boundary of $Z_C$, $ev: \cP \times \partial Z_C \to M$, the proof of Theorem \ref{TT: rel} uses the evaluation map on the whole of $Z_C$, $ev: \cP \times Z_C \to M$.
   
\subsection{Acknowledgements}
    The author thanks Amanda Hirschi, Ivan Smith and Yuhan Sun for helpful conversations, Jack Smith for helpful conversation and useful comments on an earlier draft, and the anonymous referee for useful comments and feedback. The author is supported by the Engineering and Physical Sciences Research Council [EP/W015889/1].

\section{Relationship to the Nearby Lagrangian conjecture}\label{sec: NLC}
    One of the main open problems in symplectic topology is:
    \begin{JJ}[Nearby Lagrangian Conjecture]\label{JJ: NLC}
        Any closed exact Lagrangian $L$ in a cotangent bundle $T^*Q$ is Hamiltonian isotopic to the zero-section $Q_0 \subseteq T^*Q$.
    \end{JJ}
    
    As we show in Proposition \ref{PP: NLC} below, if the Nearby Lagrangian conjecture holds, we can reduce the exact case of Theorem \ref{TT: main} to Theorem \ref{TT: old}.
    
    \begin{PP}\label{PP: NLC}
        Suppose $(M, \omega)$ is a Liouville symplectic manifold and $L$ is a closed exact Lagrangian in $M$. Let $\tilde\phi: L \to L$ be a homeomorphism that extends to a Hamiltonian homeomorphism $\phi: M \to M$. 
        
        Assume the Nearby Lagrangian Conjecture holds. Then $\tilde \phi$ is homotopic to a diffeomorphism $\tilde \psi: L \to L$ that extends to a Hamiltonian diffeomorphism $\psi: M \to M$.
    \end{PP}
    
    Using Proposition \ref{PP: NLC}, one may then apply results such as Theorem \ref{TT: main}, which only apply to Hamiltonian diffeomorphisms, to $\tilde \psi$ to deduce conclusions about the homotopy class of $\tilde \phi$.

    \begin{proof}[Proof of Proposition \ref{PP: NLC}]
        \emph{Weinstein's neighbourhood theorem} \cite[Theorem 3.33]{MD:Intro} says there is a tubular neighbourhood $U \subseteq M$ of $L$ which is symplectomorphic to the disc bundle $D^*_r L$ of some radius $r$. Let $\pi: U \to L$ be the projection map.

        By definition, we may choose a Hamiltonian diffeomorphism $\phi': M \to M$ which is $C^0$-close to $\phi$. $\phi'$ may not send $L$ to itself, but (for $\phi'$ sufficiently $C^0$-close to $\phi$), $\phi'(L)$ does lie in $U$, and $\pi \circ \phi': L \to L$ is $C^0$-close to (and hence homotopic to) $\phi$. By the Nearby Lagrangian conjecture, we may assume $L'$ is Hamiltonian isotopic to $L$ in $U$. By isotopy extension for Hamiltonians (similarly to \cite[Exercise 3.40]{MD:Intro}), there is a Hamiltonian isotopy from $L'$ to $L$ in $M$, staying entirely in the Weinstein neighbourhood $U$. Let $\theta$ be the corresponding Hamiltonian diffeomorphism of $M$, sending $L'$ to $L$. Let $\psi = \theta \circ \phi'$; this sends $L$ to $L$, and is a Hamiltonian diffeomorphism by \cite[Proposition 10.2]{MD:Intro}. We set $\tilde \psi = \psi|_L: L \to L$. It remains to show $\tilde \psi$ is homotopic to $\tilde \phi$.

        Since $\phi'$ is $C^0$-close to $\phi$, it follows that $\pi \circ \phi': L \to L$ is $C^0$-close to, and hence homotopic to, $\tilde \phi$.

        Since the Hamiltonian isotopy from $L'$ to $L$ lies in $U$, by composing with $\pi$ we obtain a homotopy from $\pi \circ \phi'$ to $\tilde \psi$. 
        
        Combining these final two homotopies, we obtain one between $\tilde \psi$ and $\tilde \phi$.
    \end{proof}
    \begin{RR}
        Let $M = T^*Q$ be a cotangent bundle. The \emph{strong nearby Lagrangian conjecture} says that the space of closed exact Lagrangians in $M$ is contractible. This would imply Conjecture \ref{JJ: main} in the case $L \subseteq M$ is the zero-section.

    \end{RR}

\section{Smooth approximation}\label{sec: approx}
    We first choose a complete Riemannian metric on $M$, and let $2\eps > 0$ be its injectivity radius. Write $d$ for the Riemannian distance function. We may assume that $L$ is a totally geodesic submanifold with respect to this metric.

    Now choose a smooth Hamiltonian diffeomorphism $\psi: M \to M$ which is sufficiently $C^0$-close to $\phi$ that for all $x$ in $M$, $d(\psi(x), \phi(x)) \leq \eps$. We write $L'$ for $\psi(L)$.\par
    Let $U$ be a tubular neighbourhood of $L$ and $\pi: U \to L$ the projection map. We choose $U$ small enough that for all $x$ in $U$,
    \begin{equation} \label{eq: pi d}
        d(x, \pi(x)) \leq \eps
    \end{equation}
    We may choose $\psi$ $C^0$-close enough to $\phi$ such that $L'$ lives in $U$, and generically so that $L$ and $L'$ intersect transversally.
    \begin{LL}\label{LL: hom}
        The maps $\phi, \pi \circ \psi: L \to L$ are homotopic.
    \end{LL}
    \begin{proof}
        For $x$ in $L$, since $\pi\circ \phi(x) = \phi(x)$, by the triangle inequality $d(\phi(x), \pi \circ \psi(x)) \leq 2\eps$. Thus there is a unique small geodesic between $\phi(x)$ and $\pi \circ \psi(x)$; since $L \subseteq M$ is totally geodesic, this geodesic stays in $L$. Moving along this geodesic for each $x$ gives the required homotopy.
    \end{proof}
\section{Moduli spaces}
    In this section, we construct moduli spaces $\cM$ and $\cP_r$ for $r \geq 0$, similarly to \cite[Sections 4.1 \& 5.3]{Porcelli} and \cite{Hofer}.\par
    Let $D$ be a smooth convex subdomain of $\C$, lying inside $\R + i[0,1]$ and containing $[-1, 1]+i[0,1]$. \par
    \begin{DD}
        We define $Z_- := D \cup \left([0, \infty) + i[0,1]\right)$. \par
        For $r \geq 0$, we define $Z_r$ to be the union
        $$Z_r := \left( D - r\right) \cup \left([-r, r]+i[0,1]\right) \cup \left(D + r\right)$$
        See below for pictures:
    \end{DD}
    
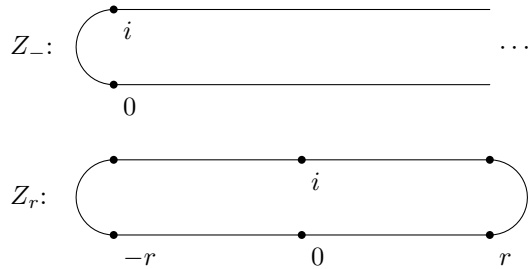
{\centering\begin{tikzpicture}[line cap=round,line join=round,>=triangle 45,x=1.0cm,y=1.0cm]        
    \draw (-1.5,+0.3) node[anchor=north west] {\parbox{0 cm}{$$ Z_-\mathchar\numexpr"6000+`:\relax  $$}};
    \draw (5,+0.3) node[anchor=north west] {\parbox{0 cm}{$$ \cdots  $$}};
    \draw (0, 0)--(5,0);
    \draw (0, -1)--(5,-1);
    \draw [shift={(0,-0.5)}] plot[domain=1.57:4.71,variable=\t]({1*0.5*cos(\t r)+0*0.5*sin(\t r)},{0*0.5*cos(\t r)+1*0.5*sin(\t r)});
    \fill [color=black] (0,-0) circle (1.5pt);
    \fill [color=black] (0,-1) circle (1.5pt);
    \draw (0,-0.5) node[anchor=north west] {\parbox{0 cm}{$$ 0 $$}};
    \draw (0,+0.5) node[anchor=north west] {\parbox{0 cm}{$$ i $$}};
    
    \draw (5,-3)-- (0,-3);
    \draw (0,-2)-- (5,-2);
    \draw (2.5,-2.5) node[anchor=north west] {\parbox{0 cm}{$$ 0 $$}};
    \draw (2.5,-1.5) node[anchor=north west] {\parbox{0 cm}{$$ i $$}};
    \draw (5,-2.5) node[anchor=north west] {\parbox{0 cm}{$$ r $$}};
    \draw (0,-2.5) node[anchor=north west] {\parbox{0 cm}{$$ -r $$}};
    \draw (-1.5,-1.7) node[anchor=north west] {\parbox{0 cm}{$$ Z_r\mathchar\numexpr"6000+`:\relax  $$}};
    \draw [shift={(0,-2.5)}] plot[domain=1.57:4.71,variable=\t]({1*0.5*cos(\t r)+0*0.5*sin(\t r)},{0*0.5*cos(\t r)+1*0.5*sin(\t r)});
    \draw [shift={(5,-2.5)}] plot[domain=-1.57:1.57,variable=\t]({1*0.5*cos(\t r)+0*0.5*sin(\t r)},{0*0.5*cos(\t r)+1*0.5*sin(\t r)});
    \begin{scriptsize}
        \fill [color=black] (0,-2) circle (1.5pt);
        \fill [color=black] (5,-2) circle (1.5pt);
        \fill [color=black] (0,-3) circle (1.5pt);
        \fill [color=black] (5,-3) circle (1.5pt);
        \fill [color=black] (2.5,-2) circle (1.5pt);
        \fill [color=black] (2.5,-3) circle (1.5pt);
    \end{scriptsize}
\end{tikzpicture}\par }
    $Z_-$ and $Z_r$ are both smooth convex subdomains of $\C$.

    \begin{DD}
        A \emph{compatible almost complex structure} on $(M, \omega)$ is a vector bundle automorphism $J$ of the tangent bundle $TM \to M$, such that $J^2 = -\Id_{TM}$, and that the tensor $\omega(\cdot, J \cdot)$ defines a $J$-invariant Riemannian metric on $M$.
    \end{DD}
    It is standard that compatible almost complex structures always exist, and so we may choose such a $J$ on $M$, which is additionally \emph{convex at infinity} (a ``niceness at infinity''-type condition, see \cite[Section 7(b)]{Seidel:book}) if $M$ is Liouville. We fix this $J$ throughout the rest of the paper.
    
    Recall a smooth map $u: \Sigma \to M$ from a Riemann surface $(\Sigma, j)$ with boundary is \emph{$J$-holomorphic} if $du \circ j = J \circ du$. The \emph{(topological) energy} of such a $J$-holomorphic $u: \Sigma \to M$ is the quantity
    $$E(u) := \int_{\Sigma} u^* \omega$$

    We take $\psi$ to be as defined in Section \ref{sec: approx}. Since $\psi$ is a Hamiltonian diffeomorphism, we may choose a Hamiltonian isotopy $\{\psi^t\}_{t \in [0,1]}$, with $\psi^0 = \Id_M$ and $\psi^1 = \psi$. We keep this fixed throughout the rest of the paper.
    \begin{DD}
        For $r \geq 0$, we define $\cP_r$ to be the space of $J$-holomorphic maps $u: Z_r \to M$ with finite energy, with moving Lagrangian boundary conditions:
        \begin{equation}\label{eq: b cond}
            u(z) \in \psi^{\operatorname{Im}(z)}(L) \, \textrm{ for all }\, z \in \partial Z_r
        \end{equation}
       
        We define $\cM$ to be the space of $J$-holomorphic maps $u: Z_- \to M$ which have finite energy, satisfying (\ref{eq: b cond}) (with $Z_r$ replaced by $Z_-$). We equip $\cM$ with the weak $C^\infty$-Whitney topology.
    \end{DD}
    In particular, if $x \in \R$, each of these moduli spaces sends $x$ to $L$ and $x+i$ to $L'$ (assuming $x, x+i$ lie in the relevant domain).
    
    Note that we do not quotient by any automorphisms in the definitions of these moduli spaces; this does not interfere with compactness, by the argument in \cite[Lemma 4.5]{Porcelli}.
    
    \cite[Lemma 4.3]{Porcelli} shows there is an a priori uniform upper energy bound on the energy of any $u \in \cM$ or $u \in \cP_r$, independent of $r$ (dependent however on the choice of $\{\psi^t\}_t$).
    
    \begin{PP}\label{PP: grom comp}
        Each $\cP_r$, as well as $\cM$,
        is compact. \par
        Furthermore for any sequence $u_n \in \cP_{r_n}$ where $r_n \to \infty$, after restricting to some subsequence, there is some $u \in \cM$ such that for any compact subset $K \subseteq Z_-$,
        $$u_n(\cdot - C_n)|_K \to u|_K$$

        converges uniformly.
    \end{PP}
    \begin{proof}
        Compactness of each $\cP_r$ is proved in \cite[Lemma 4.5]{Porcelli} (where $\cP_r$ is denoted $\cU_r$); the same argument applies also to $\cM$. Note that this requires both relative exactness of $L$ as well as cylindricity at infinity of $J$.

        The final claim is an application of \cite[Theorem 4.6.1]{McDuff-Salamon}. Explicitly in the notation of \textit{loc. cit.}, we set $\nu = n$, $\Sigma = Z_-$, $J^{n} = J$, $\Omega^n = (Z_{r_n} \cap \bC_{\mathrm{Re} < 0}) + n \subseteq \Sigma$ and $u^n = u_n(\cdot + r_n)$. Finally relative exactness of $L$ allows us to conclude the set $Z$ arising in the conclusion of \textit{loc. cit.} must be empty (since no bubbling can occur).
    \end{proof}
   
    Though we could generically choose $J$ so that these moduli spaces are transversally cut out, this would be unnecessary for our purposes.
    \begin{PP}\label{PP: cited}
        Let $C \geq 0$ and $x \in [-C, C]$, and let $ev_x: \cP_C \to L$ send $u$ to $u(x)$. Then under the orientation conditions of Theorem \ref{TT: main} for $R^* = H^*(\cdot; \bZ/2), H^*(\cdot; \bZ)$ or $K^*$, the induced map $ev^*_{x}: R^*(L) \to R^*(\cP_C)$ is injective.
    \end{PP}
    \begin{proof}
        This is proved in \cite[Theorem 3.1]{Hofer} in the case $R^*$ is (a \v{C}ech model for) $H^*(\cdot; \bZ/2)$ and \cite[Theorem 2.12]{HP} in the other cases.
    \end{proof}
    \begin{RR}
        Since $\cP_C$ can in general be extremely singular (and in particular, not necessarily homotopy equivalent to a CW complex), we must be careful with exactly what model of $R^*(\cP_C)$ we use; see \cite[Section 2.1]{HP} for a discussion about this point. However all reasonable models for $R^*(\cdot)$ agree when applied to $L$, since $L$ admits the structure of a finite CW complex, so this does not affect how we apply Proposition \ref{PP: cited}.
    \end{RR}
\section{Pinching}
    For $r \geq 0$, $u \in \cP_r$ and $x \in [-r, r]$, we define as shorthand
    \begin{equation*}
        d_u(x) := \sup_{t \in [0,1]} d\left(u(x), u(x+ti)\right)
    \end{equation*}
    \begin{PP}\label{PP: main}
        There is some $C \geq 0$ and $x_0 \in [-C, C]$ such that for all $u$ in $\cP_C$,
        \begin{equation*}
            d_u(x_0) \leq \eps
        \end{equation*}
    \end{PP}
    This should be viewed as a version of Gromov compactness; indeed we will deduce Proposition \ref{PP: main} from Proposition \ref{PP: grom comp}. First we need a lemma:
    \begin{LL}\label{LL: 1}
        There is some $\lambda > 0$ such that for all $v$ in $\cM$,
        \begin{equation*}
            d_v(\lambda) \leq \frac \eps 2
        \end{equation*}
    \end{LL}
    \begin{proof}
        Since we assumed $L$ and $L'$ intersect transversally, for fixed $v \in \cM$ this holds for $\lambda$ sufficiently large, by the standard exponential decay estimate for holomorphic curves on their strip-like ends, see e.g. \cite{Robbin-Salamon}; we may choose $\lambda>0$ large enough that this holds for all $v \in \cM$ since $\cM$ is compact.
    \end{proof}
    \begin{proof}[Proof of Proposition \ref{PP: main}]
        Let $\lambda$ be as in Lemma \ref{LL: 1}. We show that Proposition \ref{PP: main} holds for $C$ sufficiently large and $x_0 = \lambda - C$.\par
        Suppose not. There is then a sequence $C_n \to \infty$ and $u_n \in \cP_{C_n}$, such that for all $n$,
        \begin{equation}\label{eq: asmp d}
            d_{u_n}(\lambda - C_n) > \eps
        \end{equation}
        Applying Proposition \ref{PP: grom comp} to $K = \lambda + i[0,1] \subseteq Z_-$, we find some $u \in \cM$ such that
        \begin{equation}\label{eq: grom conv}
            u_n(\cdot-C_n)|_K \to u|_K
        \end{equation}
        converges uniformly. Therefore
        \begin{equation}\label{eq: conv d}
            d_{u_n}(\lambda-C_n) \to d_{u}(\lambda)
        \end{equation}
        The right hand side of (\ref{eq: conv d}) is $\leq \eps/2$ by Lemma \ref{LL: 1}, whereas the left hand side is $> \eps$ for all $n$, by (\ref{eq: asmp d}), giving a contradiction.
    \end{proof}
\section{Proof of Theorem \ref{TT: main}}
    Let $C$ and $x_0$ be as in Proposition \ref{PP: main} and let $ev_{x_0}, ev_{x_0+i}: \cP_C \to L$ be the maps which evaluate at $x_0$ and $x_0+i$ respectively
    \begin{LL}\label{LL: 3}
        The following diagram commutes up to homotopy:
        \begin{equation} \label{diag: LL3}
            \begin{tikzcd}
                \cP_C \arrow[d, "ev_{x_0}"] \arrow[dr, "ev_{x_0+i}"] &
                \\
                L \arrow[r, "\psi^1"] &
                L'
            \end{tikzcd}
        \end{equation}
    \end{LL}
    \begin{proof}
        We write down an explicit homotopy; this can be thought of as ``evaluating around the boundary of $Z_C$ from $x_0$ to $x_0+i$''. Let $\gamma: [0,1] \to \partial Z_C$ be any path from $x_0$ to $x_0+i$, and let $r_t = \operatorname{Im} \gamma(t)$ for $t \in [0,1]$. Then we define a homotopy $H: [0,1] \times \cP_C \to L'$ to send $(t, u)$ to
        $$\psi^1 \circ \left(\psi^{r_t}\right)^{-1}\circ u \circ \gamma(t)$$
        Then $H(0, \cdot) = \psi^1 \circ ev_{x_0}$ and $H(1, \cdot) = ev_{x_0 + i}$.
    \end{proof}
    Using Proposition \ref{PP: main}, we may replace $\phi$ with $\psi^1$:
    \begin{CC}\label{CC: main}
        The following diagram commutes up to homotopy:
        \begin{equation} \label{diag: main}
            \begin{tikzcd}
                \cP_C \arrow[d, "ev_{x_0}"] \arrow[dr, "ev_{x_0}"] &
                \\
                L \arrow[r, "\phi"] &
                L
            \end{tikzcd}
        \end{equation}
    \end{CC}
    \begin{proof}
        By Proposition \ref{PP: main} and (\ref{eq: pi d}), for all $u \in \cP_C$, $\pi \circ ev_{x_0+i}(u)$ and $ev_{x_0}(u)$ are of distance at most $2\eps$ from each other. Since $2\eps$ is less than the injectivity radius, the following diagram commutes, via a homotopy which follows the unique small geodesic (which lies in $L$, since $L \subseteq M$ is totally geodesic) between the two endpoints:
        \begin{equation*}
            \begin{tikzcd}
                \cP_C \arrow[swap,d, "ev_{x_0+i}"] \ar[dr, "ev_{x_0}"] &
                \\
                L' \arrow[r, "\pi"] &
                L
            \end{tikzcd}
        \end{equation*}
        Combined with Lemma \ref{LL: 3} and Lemma \ref{LL: hom}, this implies the result.
    \end{proof}

    \begin{proof}[Proof of Theorem \ref{TT: main}]
        Let $R^*$ be the relevant cohomology theory. Then by applying $R^*$ to (\ref{diag: main}) and using Proposition \ref{PP: cited}, we see that for any $\alpha \in R^*(L)$, $\phi^*(\alpha) = \alpha$.
    \end{proof}
\section{Proof of Theorem \ref{TT: rel}}
    Let $R$ be the cohomology theory $H^*(\cdot; \bZ/2), H^*(\cdot; \bZ)$ or $K^*$, and assume the relevant orientation condition from Theorem \ref{TT: rel} holds for $R$.
    
    We let $C$ and $x_0$ be as in Proposition \ref{PP: main}. We define a family of maps $\{\phi^t: M \to M\}_{t \in [0,1]}$ to be given by $\phi^t = \psi^{2t}$ for $0 \leq t \leq \frac12$, and to follow the unique small geodesic between $\psi^1(x)$ and $\phi(x)$ for $x \in M$ and $\frac 12 \leq t \leq 1$. Then $\phi^1 = \phi$ and $\phi^0$ is the identity.\par
    Consider the sweep-out map $S: L \times [0,1] \to M$, sending $(x, t)$ to $\phi^t(x)$. This sends $L \times \{0,1\}$ to $L$ and therefore a map of pairs $(L\times [0,1], L \times \{0,1\}) \to (M, L)$; applying $R^*$ and applying the suspension isomorphism gives a map
    $$\sigma: R^{*+1}(M,L) \to R^*(L)$$
    Let $\partial: R^*(L) \to R^{*+1}(M,L)$ be the boundary map.
    \begin{LL}\label{LL:6}
        The two maps
        $$\phi^*, (Id + \partial \circ \sigma): R^*(M, L) \to R^*(M, L)$$
        are equal.
    \end{LL}
    \begin{proof}
        Let $\overline M = M \cup_{L \times \{0\}} L \times [0,1]$ be the mapping cylinder of the inclusion $L \hookrightarrow M$. Then $\phi$ induces a map of pairs $(\overline M, L \times \{1\})$ to itself, and the action of this map on $R^*$ is the same as that of the action of $\phi$ on $(M, L)$.\par
        We define another map of pairs $\theta$ from $(\overline M, L \times \{1\})$ to itself: $\theta$ sends $x \in M$ to $x$ and $(x, t) \in L \times [0,1]$ to $(\phi^t(x),t)$. By construction, the map induced by $\theta$ on $R^*(M,L)$ is $Id + \partial \circ \sigma$, so it remains to show that $\theta$ and $\phi$ are homotopic as maps of pairs.\par
        Define $H: \overline M \times [0,1]_s \to \overline M$ to send $(x, s) \in M \times [0,1]$ to $\phi^s(x)$ and $(x, t, s) \in L \times [0,1]^2$ to
        $$\left(\phi^{\min \{s+t,1\}}(x), t\right)$$
        Then $H(\cdot, 0) = \theta$, $H(\cdot, 1) = \phi$ and $H(\cdot, t)$ sends $L \times \{1\}$ to itself for all $t$; together these imply that $H$ is the desired homotopy.
    \end{proof}
    We define a map $T: \cP_C \times [0,1] \to M$ as follows. Let $\gamma: [0,1] \to \partial Z_C$ be a path from $x_0$ to $x_0+i$ and let $r_t = \operatorname{Im}\gamma(t)$. Then we define $T(u, t)$ to be $u \circ \gamma(2t)$ for $0 \leq t \leq \frac 12$ and to follow the unique small geodesic between $ev_{x_0+i}(u)$ and $\phi \circ (\psi^1)^{-1} \circ ev_{x_0+i}(u)$ for $\frac 12 \leq t \leq 1$.\par
    We further define a map $F: \cP_C \times [0,1] \to L \times [0,1]$ by
    \begin{equation*}
        F(u,t) := \begin{cases}
            \left(\left(\psi^{r_{2t}}\right)^{-1} \circ u \circ \gamma(2t), t\right) &
            \textrm{ if } 0 \leq t \leq \frac12\\
            \left(\left(\psi^1\right)^{-1} \circ ev_{x_0+i}(u), t\right) &
            \textrm{ if } \frac12 \leq t \leq 1
        \end{cases}
    \end{equation*}
    \begin{LL}\label{LL:6.5}
        The following diagram commutes up to homotopy (relative to $\cP_C \times \{0,1\}$):
        \begin{equation*}
            \begin{tikzcd}
                \cP_C \times [0,1] \arrow[d, "F"] \arrow[dr, "T"] &
                \\
                L \times [0,1] \arrow[r, "S"] &
                M
            \end{tikzcd}
        \end{equation*}
    \end{LL}
    \begin{proof}
        We define a map $S': L \times [0,1] \to M$ to send $(x, t)$ to $\psi^{r_{2t}}(x)$ if $0 \leq t \leq \frac12$ and to follow the unique small geodesic between $\psi^1(x)$ and $\phi(x)$ for $\frac 12 \leq t \leq 1$. By construction, the following diagram commutes:
        \begin{equation*}
            \begin{tikzcd}
                \cP_C \times [0,1] \arrow[d, "F"] \arrow[dr, "T"] &
                \\
                L \times [0,1] \arrow[r, "S'"] &
                M
            \end{tikzcd}
        \end{equation*}
        $S'$ and $S$ are homotopic (relative to $L \times \{0,1\}$) since they are the same up to a reparametrisation of $[0,\frac12]$.
    \end{proof}
    \begin{LL}\label{LL:7}
        The sweep-out map $\sigma$ vanishes.
    \end{LL}
    \begin{proof}
        Proposition \ref{PP: cited} implies that the map
        \begin{equation*}
            R^*(L \times [0,1], L \times \{0,1\}) \to R^*(\cP_C \times [0,1], \cP_C \times \{0,1\})
        \end{equation*}
        is injective, so it suffices to show that $T$ is homotopic (through maps sending $\cP_C \times \{0,1\}$ to $L$) to a map landing inside the tubular neighbourhood $U$ of $L$. Such a homotopy exists by homotoping $\gamma$ (relative to its endpoints) through $Z_C$ to the straight-line path from $x_0$ to $x_0+i$, and using a map defined with the same formula (with respect to this family of paths instead of $\gamma$) as $T$. Note that here we apply Proposition \ref{PP: main}.
    \end{proof}
    \begin{proof}[Proof of Theorem \ref{TT: rel}]
        Follows from Lemmas \ref{LL:6} and \ref{LL:7}.
    \end{proof}
\bibliography{Refs.bib}{} \bibliographystyle{abbrv}

\begin{thebibliography}{10}

\bibitem{Jack+}
M.~Augustynowicz, J.~Smith, and J.~Wornbard.
\newblock Homological {L}agrangian monodromy for some monotone tori, 2022.
\newblock Accepted, to appear in Quantum Topology. Available at arXiv:2201.10507.

\bibitem{BHS}
L.~Buhovsky, V.~Humili\`ere, and S.~Seyfaddini.
\newblock A {$C^0$} counterexample to the {A}rnold conjecture.
\newblock {\em Invent. Math.}, 213(2):759--809, 2018.

\bibitem{CP}
S.~Courte and N.~Porcelli.
\newblock On the parametrised whitehead torsion of families of nearby lagrangian submanifolds, 2025.
\newblock Preprint, available at arXiv:2506.06110.

\bibitem{Eliashberg}
Y.~M. Eliashberg.
\newblock A theorem on the structure of wave fronts and its application in symplectic topology.
\newblock {\em Funktsional. Anal. i Prilozhen.}, 21(3):65--72, 96, 1987.

\bibitem{Floer}
A.~Floer.
\newblock Morse theory for {L}agrangian intersections.
\newblock {\em J. Differential Geom.}, 28(3):513--547, 1988.

\bibitem{Fukaya-Ono}
K.~Fukaya and K.~Ono.
\newblock Arnold conjecture and {G}romov-{W}itten invariant.
\newblock {\em Topology}, 38(5):933--1048, 1999.

\bibitem{Gromov}
M.~Gromov.
\newblock {\em Partial differential relations}, volume~9 of {\em Ergebnisse der Mathematik und ihrer Grenzgebiete (3) [Results in Mathematics and Related Areas (3)]}.
\newblock Springer-Verlag, Berlin, 1986.

\bibitem{HP}
A.~Hirschi and N.~Porcelli.
\newblock Lagrangian intersections and cuplength in generalised cohomology, 2022.
\newblock Preprint, available at arXiv:2211.07559. To appear in Mathematical Research Letters.

\bibitem{Hofer}
H.~Hofer.
\newblock Lusternik-{S}chnirelman-theory for {L}agrangian intersections.
\newblock {\em Ann. Inst. H. Poincar\'{e} Anal. Non Lin\'{e}aire}, 5(5):465--499, 1988.

\bibitem{Hu-Lalonde-Leclercq}
S.~Hu, F.~Lalonde, and R.~Leclercq.
\newblock Homological {L}agrangian monodromy.
\newblock {\em Geom. Topol.}, 15(3):1617--1650, 2011.

\bibitem{HJL}
V.~Humilière, A.~Jannaud, and R.~Leclercq.
\newblock Essential loops in completions of hamiltonian groups, 2024.
\newblock Preprint, available at arXiv:2311.12164.

\bibitem{Jannaud:free}
A.~Jannaud.
\newblock Free subgroup of the $c^0$ symplectic mapping class group, 2022.
\newblock Preprint, available at arXiv:2211.05570.

\bibitem{Keating:Dehn}
A.~M. Keating.
\newblock Dehn twists and free subgroups of symplectic mapping class groups.
\newblock {\em J. Topol.}, 7(2):436--474, 2014.

\bibitem{Liu-Tian}
G.~Liu and G.~Tian.
\newblock Floer homology and {A}rnold conjecture.
\newblock {\em J. Differential Geom.}, 49(1):1--74, 1998.

\bibitem{McDuff-Salamon}
D.~McDuff and D.~Salamon.
\newblock {\em {$J$}-holomorphic curves and symplectic topology}, volume~52 of {\em American Mathematical Society Colloquium Publications}.
\newblock American Mathematical Society, Providence, RI, second edition, 2012.

\bibitem{MD:Intro}
D.~McDuff and D.~Salamon.
\newblock {\em Introduction to symplectic topology}.
\newblock Oxford Graduate Texts in Mathematics. Oxford University Press, Oxford, third edition, 2017.

\bibitem{Porcelli}
N.~Porcelli.
\newblock Families of relatively exact {L}agrangians, free loop spaces and generalised homology.
\newblock {\em Selecta Math. (N.S.)}, 30(2):Paper No. 21, 53, 2024.

\bibitem{Robbin-Salamon}
J.~W. Robbin and D.~A. Salamon.
\newblock Asymptotic behaviour of holomorphic strips.
\newblock {\em Ann. Inst. H. Poincar\'{e} C Anal. Non Lin\'{e}aire}, 18(5):573--612, 2001.

\bibitem{Ruan}
Y.~Ruan.
\newblock Virtual neighborhoods and pseudo-holomorphic curves.
\newblock In {\em Proceedings of 6th {G}\"okova {G}eometry-{T}opology {C}onference}, volume~23, pages 161--231, 1999.

\bibitem{Seidel:rep}
P.~Seidel.
\newblock {$\pi_1$} of symplectic automorphism groups and invertibles in quantum homology rings.
\newblock {\em Geom. Funct. Anal.}, 7(6):1046--1095, 1997.

\bibitem{Seidel:knotted}
P.~Seidel.
\newblock Lagrangian two-spheres can be symplectically knotted.
\newblock {\em J. Differential Geom.}, 52(1):145--171, 1999.

\bibitem{Seidel:book}
P.~Seidel.
\newblock {\em Fukaya categories and {P}icard-{L}efschetz theory}.
\newblock Zurich Lectures in Advanced Mathematics. European Mathematical Society (EMS), Z\"urich, 2008.

\bibitem{Jack}
J.~Smith.
\newblock Hamiltonian isotopies of relatively exact {L}agrangians are orientation-preserving.
\newblock {\em Adv. Geom.}, 24(4):505--506, 2024.

\bibitem{Mei-LinYau}
M.-L. Yau.
\newblock Monodromy and isotopy of monotone {L}agrangian tori.
\newblock {\em Math. Res. Lett.}, 16(3):531--541, 2009.

\end{thebibliography}
\Addresses
\end{document}